\documentclass[preprint,11pt]{elsarticle}

\makeatletter
\def\ps@pprintTitle{%
 \let\@oddhead\@empty
 \let\@evenhead\@empty
 \def\@oddfoot{}%
 \let\@evenfoot\@oddfoot}
\makeatother
\usepackage[titletoc,title]{appendix}
\makeatletter
  \@addtoreset{chapter}{part}
  \@addtoreset{@ppsaveapp}{part}
\makeatother

\usepackage[utf8]{inputenc}
\usepackage{array,graphicx}
\usepackage{pifont}
\usepackage{natbib}[numbers]
\setcitestyle{numbers, open={[},close={]}}
\usepackage{amsfonts}
\usepackage{soul}
\usepackage{amsmath,amsthm}
\usepackage{amssymb}
\usepackage{booktabs}
\usepackage{multicol}
\usepackage{enumitem}
\usepackage{algorithm,algpseudocode}
\usepackage[dvipsnames]{xcolor}
\usepackage{setspace}
\floatname{algorithm}{}
\let\Algorithm\algorithm
\renewcommand\algorithm[1][]{\Algorithm[#1]\setstretch{1.2}}

\makeatletter
\renewcommand{\fnum@algorithm}{\fname@algorithm}
\makeatother

\usepackage{euscript}
\usepackage{subfigure}
\usepackage{verbatim} 
\usepackage{xspace} 
\usepackage{verbatim}
\usepackage{xcolor}
\usepackage{standalone}
\usepackage{pgfplots}
\pgfplotsset{compat=1.14}
\usepackage{pgf,tikz}
\pgfkeys{/pgf/fpu=false}
\usetikzlibrary{arrows}
\relax
\setlist{leftmargin=2.5mm}
\usepackage{xspace} 
\usepackage{hyperref}
\usepackage{comment}
\usepackage[pagewise]{lineno}
\newtheorem{proposition}{Proposition}[section]
\newtheorem{remark}{Remark}[section]
\newtheorem{example}{Example}[section]

\newtheorem{problem}{Problem}[section]

\hypersetup{%
    pdfpagemode={UseOutlines},
    bookmarksopen,
    pdfstartview={FitH},
    colorlinks=false,
    linkcolor={blue},
    citecolor={blue},
    urlcolor={green},
    pdfauthor=G.Dragotto
  }
\usepackage{url,xspace}

\newcommand{\Mod}[1]{\ (\mathrm{mod}\ #1)}

\def\Colorato{Color}
\def\Rendered{}
\ifx\Rendered\undefined
\newcommand{\Includiamolo}[2]{\includestandalone[#1]{Resources/#2}}
\else
\newcommand{\Includiamolo}[2]{\includegraphics[#1]{Resources/render/#2.png}}
\fi

\begin{document}
\pagestyle{plain}

\begin{frontmatter}

  \title{Merging Combinatorial Design and Optimization: the Oberwolfach Problem}
  \date{}

  \author[1]{Fabio Salassa}
  \author[2,1]{Gabriele Dragotto \corref{cor1}}
  \author[4]{Tommaso Traetta}
  \author[5]{Marco Buratti}
  \author[1,3]{Federico Della Croce}

  \cortext[cor1]{Corresponding author.}

  \address[1]{\footnotesize  {\tt \{fabio.salassa, federico.dellacroce\}@polito.it } \\ Dipartimento di Ingegneria Gestionale e della Produzione, Politecnico di Torino (Italy)}
     \address[3]{CNR, IEIIT, Torino, Italy}
  \address[2]{\footnotesize  {\tt gabriele.dragotto@polymtl.ca }\\ Canada Excellence Research Chair in Data Science for Real-time Decision-making, \'Ecole Polytechnique de Montréal (Canada)}
  \address[4]{\footnotesize   {\tt tommaso.traetta@unibs.it } \\ DICATAM, Università degli Studi di Brescia (Italy)}
   \address[5]{\footnotesize  {\tt marco.buratti@unipg.it } \\ Dipartimento di Matematica e Informatica, Università degli Studi di Perugia (Italy)}

\begin{abstract}

The Oberwolfach Problem $OP(F)$ -- posed by Gerhard Ringel in 1967 -- is a paradigmatic Combinatorial Design problem asking whether the complete graph $K_v$ decomposes into edge-disjoint copies of a $2$-regular graph $F$ of order $v$. 
In this paper, we provide all the necessary equipment to generate solutions to $OP(F)$ for relatively small orders by using the so-called difference methods. From the theoretical standpoint, we present new insights on the combinatorial structures involved in the solution of the problem. Computationally, we provide a full recipe whose base ingredients are advanced optimization models and tailored algorithms. This algorithmic arsenal can solve the $OP(F)$ for all possible orders up to $60$ with the modest computing resources of a personal computer. The new $20$ orders, from $41$ to $60$, encompass $241200$ instances of the Oberwolfach Problem, which is 22 times greater than those solved in previous contributions.
\end{abstract}  
  
\end{frontmatter}
\newcommand\OneRot{$1$-{\em rotational}\xspace}
\newcommand\TwoRot{$2$-{\em rotational}\xspace}
\newcommand\BLPName{Binary Labeling Problem}
\newcommand\BLP{(BLP)\xspace}
\newcommand\BLPNameEven{Even Binary Labeling Problem}
\newcommand\BLPEven{(eBLP)\xspace}
\newcommand\GLPName{Group Labeling Problem}
\newcommand\GLP{(GLP)\xspace}
\newcommand\GLPNameEven{Even Group Labeling Problem}
\newcommand\GLPEven{(eGLP)\xspace}
\newcommand\FStar{($F^*$ LP)}
\newcommand\FStarName{$F^*$ Labeling}
\newcommand\Prop{\textit{Proposition}\;}
\newcommand\Rem{\textit{Remark}\;}

\section{Introduction}
\label{sec:intro}

Gerhard Ringel proposed the Oberwolfach Problem ($OP$) for the first time in $1967$ \citep{Guy1971}, while attending a conference at the Mathematical Research Institute of Oberwolfach, in Germany.
In conferences held at the Institute, participants usually dine together in a room with circular tables of different sizes, and each participant has an assigned seat. Ringel asked whether there exists a seating arrangement for an odd number $v$ of people and $(v-1)/2$ meals so that all pairs of participants are seated next to each other exactly once. 

Formally, given a spanning $2$-regular subgraph (a \textit{$2$-factor}) $F$ of $K_v$ 
(the \textit{complete graph} of $v$ vertices), the Oberwolfach problem $OP(F)$ asks whether $K_v$ with $v$ odd decomposes into $(v-1)/2$ edge-disjoint copies of $F$. We describe $F$ as $[\dots,^{m_i}c_i, \dots]$ -- a list of cycles with their multiplicity -- where each $m_i$ is the multiplicity of $c_i$-length cycles in $F$. For the sake of simplicity, we omit any $c_i$ for which $m_i=0$, and we remove the superscript when $m_i=1$ for a given $i$.

In 1979, \citet{Huang1979} extended the problem to the case where $v$ is even. Although $OP$ has drawn interest, and much progress has been made over the past few years (see, for instance, \cite{Bryant2011, Bryant2009, Buratti2008, Asymptotic, Hilton2001, Ollis2009, TRAETTA2013984}), a complete solution has yet to be found. 
A survey of the most relevant results on this problem, updated to 2006, can be found in \cite{Colbourn1996}.

Solutions to $OP$ can often be found by focusing on those having symmetries with a particular action on the vertex set. By knowing the structure of these solutions, the problem of finding edge-disjoint $2$-factors turns into finding a few well-structured $2$-factors.
The so-called \textit{difference methods} -- a family of algebraic tools -- set the rules to construct such well-structured $2$-factors.
\textit{Difference methods} were introduced for the first time by \citet{Anstice1852} to generalize solutions to Kirkman's 15 schoolgirls problem, one of the paradigmatic problems in Combinatorial Design Theory. Arranging seats around tables is not new for Operations Research as well. 
\citet{Garcia2014}, for instance, introduced a table placement problem aiming to maximize a measure of social benefit. 

The baseline of this work is the contribution of \citet{Deza2010}. There, authors solved $OP$ for $18\leq v \leq 40$, modeling difference methods with undisclosed algorithms and tests which were carried out on a high-performance computing cluster \citep{SHARCNET}. 

In this paper, we provide all the necessary equipment to generate solutions to an Oberwolfach problem of a relatively small order; not only on the theoretical side but, most importantly and differently from \citet{Deza2010}, also on the computational side. Indeed we give a full recipe whose base ingredients are advanced optimization algorithms rather than an exhaustive search. These algorithms allowed us to rapidly obtain the desired solutions for all possible orders $v\le60$ with a personal computer which, we point out, is a much lower performance threshold compared to a high-performance computing cluster of \citet{Deza2010} in 2008. We also recall that the number of generated solutions increases with the order of the problem, for instance, from $1756$ partitions of $v=40$ to $33552$ partitions of $v=60$.  

Our approach blends combinatorial design theory with optimization and computation paradigms. We model difference methods as Constraint Programming ($CP$) problems, and leverage on state-of-the-art algorithms to find the combinatorial solutions.
We were able to generate complete solutions for $OP$ when $v \le 60$. Our approach solves a generic instance within 5 seconds on a standard computer, compared to the need for a high-performance computing cluster \citep{Deza2010}.
The extensive computational testing allowed us to derive new theoretical results for $OP$, in particular, a new necessary condition was detected on the existence of \OneRot solutions. Moreover, an Integer Programming ($IP$) model verifies the non-existence of solutions for $OP(^23,5)$.

In a nutshell, the above optimization tools enabled us to solve large $OP$ instances in limited CPU times and at the same time to derive theoretical results for general classes of instances. We believe such an approach could be generalized to a broader class of Combinatorial Design problems.

The paper proceeds as follows. In Section $2$, since this work is at the intersection of two distinct domains, Combinatorial Design Theory and Combinatorial Optimization, we introduce a standard tool pertaining to the former (difference methods) by means of an illustrative example. Section $3$ presents how to construct well-structured $2$-factors for the solution of  $OP(F)$. Section $4$ shows how to translate into $CP$ models the findings of section $3$. Section $5$ provides the outcome of the experimental testing. Section $6$ concludes the paper with final remarks.

\newcommand{\e}[2]{\lfloor #1, #2\rfloor}

\section{Difference methods and \texorpdfstring{$OP(F)$}{OP(F)}: an illustrative example}
\label{sec:example}

\textit{Difference methods} exploit the symmetries of a $2$-factorization and tell us how to construct one well-structured $2$-factor which yields, by \textit{translation}, the complete set of $2$-factors giving a solution to $OP(F)$. To explain it in the context of the Ringel's informal formulation, we can construct, for instance, the first meal seating arrangement (the desired well-structured $2$-factor) and \textit{derive/translate} from it the remaining ones.
In the following example, we provide a well-structured $2$-factor for $OP(3,6)$, and show how the remaining meals can be \textit{derived} starting from it.

\begin{figure}[!ht]
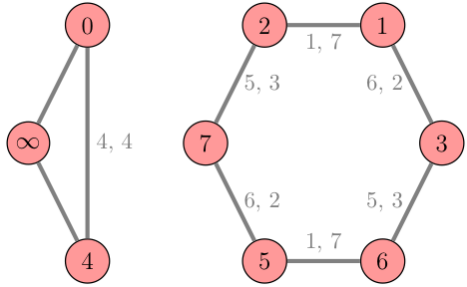

\centering
 \Includiamolo{width=0.5\textwidth}{OP_DifferenceFamilies\Colorato}
      \caption{A first meal arrangement for $OP(F=[3,6])$. }
    \label{img:DifferenceMethods}
\end{figure}

Figure \ref{img:DifferenceMethods} depicts the seating arrangement of the first meal  (see Section \ref{sec:1rot}, \Prop \ref{prop:1rot}) of $OP(F)$, where $F$ contains two cycles (dinner tables) of size 3 and 6, thus $F=[3,6]$ and $v=9$.

We label the vertices of $F$ with the elements of  $\mathbb{Z}_8\ \cup\ \{\infty\}$ and
for each edge incident with two vertices different from $\infty$, say $i$ and $j$, 
the two differences $i - j$ and $j - i$ (both $\mod (v - 1)$) have to be calculated.
For instance, if we consider the vertices labeled 2 and 1 in Figure \ref{img:DifferenceMethods}, 
we obtain the differences 1 and $-1\equiv 7 \pmod{8}$.
The list $\Delta F$ of all possible differences between adjacent vertices of $F$, different from $\infty$, contains every element in $\{1,2,\ldots, 7\}$ with multiplicity $2$. 
Furthermore, $F+4 = F$, where $F+4$ is the graph obtained from $F$ by adding 4 to every vertex but $\infty$.
 In other words, we have found a vertex labeling of $F$ such that $\Delta F$ contains every element in $\{1,2,\ldots, 7\}$ with multiplicity $2$, and $F+4=F$. These are the two crucial conditions which guarantee that $F$ is the sought-after $2$-factor (see \Prop \ref{prop:1rot}) which will generate a solution to $OP(F)$.
Indeed, the set $\mathcal{F} = \{F, F+1, F+2, F+3\}$ of all distinct translates of $F$ 
(see Figure \ref{img:DifferenceMethods2}) gives us a set of edge-disjoint copies of $F$ which decompose $K_9$, that is, a solution to $OP(F)$.

\begin{figure}[!ht]
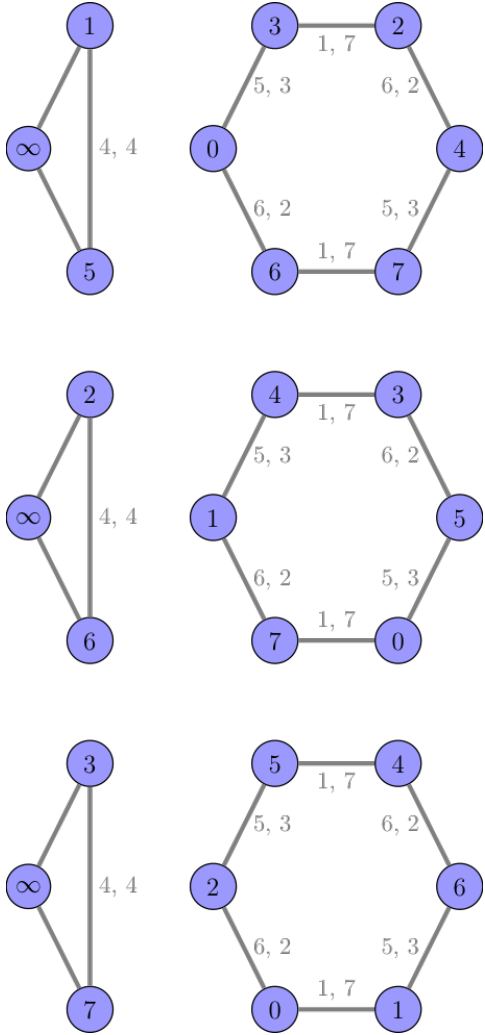

\centering
   \Includiamolo{width=0.5\textwidth}{OP_DifferenceFamilies2\Colorato}
      \caption{The remaining meals arrangements for $OP(F=[3,6])$. }
    \label{img:DifferenceMethods2}
\end{figure}

In the following section we provide conditions to find a well-structured $2$-factor $F$ which guarantee that all its distinct translates yield a solution to  OP$(F)$. 
In Section \ref{sec:solOP}, these conditions are then reformulated as $CP$ models to be tackled by a solver to generate solutions (i.e. first meal arrangements).

\section{Constructions of well-structured \texorpdfstring{$2$-factors}{2-factors}}
\label{sec:ober}

A graph has a 2-factorization if and only if it is regular and of even degree, as \citet{Petersen1891} shows.
However, given a particular 2-factor $F$, if we ask for a $2$-factorization whose factors are all isomorphic to $F$, then the problem becomes much harder.
Our focus is on constructing 2-factorizations of $K_v^*$ which is the complete graph $K_v$ of order $v$ when $v$ is odd, or $K_v-I$, i.e. the complete graph $K_v$ minus the 1-factor $I$, when $v$ is even. 
Given a $2$-factor $F$ of order $v$, the Oberwolfach problem on $F$ ($OP(F)$) asks for a $2$-factorization of $K_v^*$ into copies of $F$. 

A solution to $OP(F)$ exists whenever the order of $F$ is less than $40$ \citep{Deza2010}, except only when
$F \in \{[\;^2 3], [^4 3]$, $[4, 5], [^2 3 , 5]\}$. These are the only known cases in which the problem is not solvable. We point out that even though  \citet{Piotrowski1991} self-cites for a computer-based proof of the non-existence of a solution to $OP(^23,5)$, no published proof exists. 
$OP(F)$ has also been solved when $F$ is a uniform $2$-factor (i.e., $F$ consists of cycles of the same length) \citep{Alspach1985, Alspach1989, Hoffman1991}, when $F$ is bipartite (i.e., $F$ contains only cycles of even lengths) \citep{Alspach1985, Bryant2011}, when $F$ has exactly two cycles \citep{TRAETTA2013984}, or for an infinite family of prime orders \citep{Bryant2009}. In addition, \citep{Buratti2005, Buratti2008, Buratti2012} studied solutions having symmetries with a prescribed action on the vertices, 
and existence for sufficiently large $v$ can be found in \citep{Asymptotic}. Furthermore, a solution to the infinite variant of the Oberwolfach problem can be found in \citep{Costa_Complete_2020}.
However, the problem is still open in general, and \citep[Section VI.12]{Colbourn1996} provides a detailed survey on this subject, updated to 2006. 
\subsection{\texorpdfstring{\OneRot}{1-rotational} solutions to \texorpdfstring{$OP(F)$}{OP(F)}}
\label{sec:oberdiff}
\label{sec:1rot}\
\citet{Buratti2008} construct \OneRot solutions in the odd case, that is, $2$-factorizations of $K_v$, with $v$ odd, with a well-behaved automorphism group. 
Let $v=2n+1$ be a positive integer, let $\mathbb{Z}_{2n}$ denote the group of integers modulo $2n$, and set 
$V = \mathbb{Z}_{2n}\ \cup\ \{\infty\}$.
The list of differences of a subgraph $\Gamma$ of $K_V$ is the multiset $\Delta \Gamma$ of all possible differences 
between pairs of adjacent vertices of $\Gamma$, excluding the vertex $\infty$, namely:
\begin{eqnarray}
  \Delta \Gamma = \big\{ x-y \;\big|\;
  \e{x}{y}\in E(\Gamma\setminus\{\infty\})\big\}.
  \label{eq:1RotDiff}
\end{eqnarray}
We use the notation $\lfloor x,y \rfloor$ to denote an edge between the nodes $x$ and $y$.
Also, for every $g\in \mathbb{Z}_{2n}$, we denote by $\Gamma+g$ the graph with vertex set $V$ obtained from $\Gamma$ by replacing each vertex $x\neq \infty$ with $x+g$. 

The following result, proven in \citep{Buratti2008}, provides conditions which guarantee the existence of a solution to 
$OP(F)$. 
\begin{proposition}\label{prop:1rot}
Let $F$ be a $2$-regular graph satisfying the following properties:
\begin{enumerate}[leftmargin=15pt]
  \item $V(F) = \mathbb{Z}_{2n}\ \cup\ \{\infty\}$,
  \item $\Delta F \supset \mathbb{Z}_{2n}\setminus\{0\}$,
  \item $F+n=F$.
\end{enumerate}
Then $\mathcal{F}=\{F+g\mid g\in \mathbb{Z}_{2n}\}$ is a solution of $OP(F)$.
\end{proposition}

A factorization $\mathcal{F}$ of $K_{2n+1}$ constructed as in Proposition \ref{prop:1rot} is called \OneRot, since 
the permutation group $G=\{\tau_g \mid g\in \mathbb{Z}_{2n}\}$ of $V$, where $\tau_g$ fixes $\infty$ and maps 
$x\in \mathbb{Z}_{2n}$ to $x + g$, is an automorphism group of $\mathcal{F}$ whose action on $V\setminus\{\infty\}$ creates only one orbit.

%\begin{figure}[!h]
%\centering
%  \includestandalone[width=\textwidth]{resources/OP_Example_1Rot\Colorato}
%      \caption{$n=8$ distinct translations generate a \textit{1-rotational} solution for $OP(F=[3,6,8])$.}
%    \label{img:FStar}
%\end{figure}

In \citep[Proposition 2.5]{Buratti2015} it is shown that if there exists 
a $2$-regular graph $F=[\ell_1, \ell_2 \ldots,\; \ell_s]$ of order $2n+1$ 
satisfying the assumptions of 
Proposition \ref{prop:1rot}, then there exists a solution to 
$OP([\ell_1+1, \ell_2 \ldots,\; \ell_s])$ where $\ell_1$ is the length of the cycle of 
$F$ through $\infty$. The following result weakens this condition by showing that 
$\ell_1$ can be the length of any cycle of $F$ provided that it contains a suitable 
difference.

\begin{proposition}\label{prop:2pyr} 
Let $F=[\ell_1, \ell_2 \ldots,\; \ell_s]$ be a $2$-regular graph of order $2n+1$ satisfying the assumptions of Proposition \ref{prop:1rot}.
If $C$ is an $\ell_1$-cycle of $F$ such that $\Delta C$ contains an element of $\mathbb{Z}_{2n}$ of order $2 \pmod{4}$,
then there exists a solution
to $OP([\ell_1+1, \ell_2 \ldots,\; \ell_s])$.
\end{proposition}
\begin{proof} Let $C=(c_1, c_2,\ldots, c_{\ell_1})$ be the $\ell_1$-cycle of $F$ whose list of differences contains, by assumption, an element $x$ of order $u\equiv 2 \pmod{4}$.
Without loss of generality, we can assume that $x=c_1 -c_{2}$. 

Let $\mathcal{G}=\big\{2x\cdot i + j\mid i\in\{0,1, \ldots, u/2-1\}, j\in\{0,1,\ldots,  2n/u-1\}\big\}$ and
recall that, by the definition of order,  $u$ is the smallest positive integer such that $xu \equiv 0 \pmod{2n}$; hence $xu/2 \equiv n \pmod{2n}$. Therefore, it is not difficult to check that
\begin{equation}\label{partitions}
  \text{both $\{\mathcal{G}, \mathcal{G}+x\}$ and $\{\mathcal{G}, \mathcal{G}+n\}$ are partitions of $\mathbb{Z}_{2n}$}.
\end{equation}
Considering that $F$ satisfies the assumptions of Proposition \ref{prop:1rot}, we have that
$\mathcal{F}=\{F+g\mid g\in \mathbb{Z}_{2n}\}$ is a solution of $OP(F)$. 
By taking into account \eqref{partitions} and recalling that $F+n = F$, it follows that $\mathcal{F}=\{F+g\mid g\in \mathcal{G}\}$.

Now set $C' =(c_1, \infty', c_2, \ldots, c_{\ell_1})$, where $\infty'\not\in \mathbb{Z}_{2n}\cup\{\infty\}$, and let $F'$ be the $2$-regular graph of order $2n+1$ obtained from $F$ by replacing $C$ with $C'$. 
Finally, set $\mathcal{F'}=\{F'+g \mid g\in \mathcal{G}\}$, where 
$F'+g$ is the graph obtained from $F'$ by replacing each vertex $v\not\in \{\infty, \infty'\}$ with $v+g$,
and set $V=\mathbb{Z}_{2n} \cup \{\infty, \infty'\}$.

In order to prove that $OP(F')$ has a solution, we will show that 
$\mathcal{F'}$ is a $2$-factorization of $K^*_{2n+2} = K_{V}-I$, 
where $I=\{\lfloor\infty, \infty'\rfloor\}\cup \{\lfloor g, x+g\rfloor\mid g\in\mathcal{G}\}$.
Note that \eqref{partitions} guarantees that $I$ is a $1$-factor of $K_{V}$.
Also, since $F'$ contains all the edges of $F$ except only for $\lfloor c_1, c_2\rfloor$, and considering that
\[
\{\lfloor c_2+g, c_1+g\rfloor\mid g\in\mathcal{G}\} = \{\lfloor g, x+g\rfloor\mid g\in\mathcal{G}\} \subset I,
\] 
it follows that
every edge of $K_{V}-I$ of the form $\lfloor a,b\rfloor$ with $a\neq \infty' \neq b$ belongs to exactly one cycle of 
$\mathcal{F'}$. Finally, again by \eqref{partitions} we have that 
$\big\{\lfloor \infty', c_1+g\rfloor, \lfloor \infty', c_2+g \rfloor \mid g\in\mathcal{G}\big\} = 
 \big\{\lfloor \infty', b\rfloor \mid b\in \mathbb{Z}_{2n}\big\}$, therefore  
every edge of $K^*_{2n+2}$ of the form $\lfloor \infty', b \rfloor$  with $b\in \mathbb{Z}_{2n}$ belongs to exactly
one cycle of $\mathcal{F'}$. Hence, 
$\mathcal{F'}$ is a $2$-factorization of $K^*_{2n+2}$ and this completes the proof. 
\end{proof}

We now provide two necessary conditions for the existence of a $2$-regular graph satisfying the assumptions of 
Proposition \ref{prop:1rot}.
\begin{proposition} \label{prop:1Rotnecessary1}
  If $F=[\;^{a_1}\ell_1, \;^{a_2} \ell_2, \ldots,\;^{a_s} \ell_s]$ is a $2$-regular graph of odd order $2n+1$ satisfying the assumptions of Proposition \ref{prop:1rot}, then
  \begin{equation}\label{eq:1Rotnecessary1}
    |\{i \mid a_i \ell_i\; \text{is odd}\}| = 1.
  \end{equation}
\end{proposition}
\begin{proof}
In \cite[Proposition 3.4]{Buratti2008}, the authors show that
\begin{enumerate}[leftmargin=15pt]
  \item the cycle of $F$ passing through $\infty$ has odd length, and
  \item if $C$ is any other cycle of $F$ such that $C + n = C$, then $C$ has even length.

\end{enumerate}
Therefore, if $C$ is an odd length cycle of $F$ not passing through $\infty$, then $C\neq C+n \in F$. In other words, if $a_i$ denotes the number of cycles of $F$ of odd length $\ell_i$, then $a_i$ is even, unless $\ell_i$ is the length of the cycle through $\infty$ and the assertion follows.
\end{proof}

\begin{remark} \label{rotatedcycles} Let $C$ be a cycle with vertices in $V=\mathbb{Z}_{2n}\cup\{\infty\}$ such that
$C= C+n$. It is not difficult to check that $C$ has one of the following forms:
\begin{enumerate}[leftmargin=15pt]
  \item[$(a)$] $C=(\infty, c_1, \ldots, c_{m-1}, c_{m}, c_{m}+n, c_{m-1}+n, \ldots, c_1 + n)$, \label{remark:rotated_a}
  \item[$(b)$] $C=(c_1, \ldots, c_{m-1}, c_{m}, c_{m}+n, c_{m-1}+n, \ldots, c_1 + n)$, \label{remark:rotated_b}
  and $\infty\not\in V(C)$,
  \item[$(c)$] $C=(c_1, c_2, \ldots, c_m, c_1 + n, c_2 + n, \ldots, c_{m}+n)$, and $\infty\not\in V(C)$.  \label{rotated_c}
\end{enumerate}
In the first two cases, the translation by $n$ acts on $C$ as a reflection, while in the latter case such an action is a rotation. In \cite[Proposition 3.7]{Buratti2008}, it is shown in particular that a $2$-factor of $K_V$ satisfying the assumptions of Proposition \ref{prop:1rot} contains exactly one cycle on which the translation by $n$ acts as a reflection, which then coincides with the cycle through $\infty$. Therefore, any cycle $C$ of $F$ fixed by $n$ and not passing through $\infty$ has the same form as in $(c)$.
\end{remark}

The following result can be seen as a generalization  of \cite[Theorem 3.2]{Buratti2009}.

\begin{proposition} \label{prop:1Rotnecessary2}
   Let $F$ be a $2$-regular graph of order $2n+1$ and let $r$ denote the number of cycles in $F$ 
   of even length.
   If $F$ satisfies the assumptions of Proposition \ref{prop:1rot} and its cycle 
   passing through $\infty$ has length $3$, then 
   either $n\equiv 0\pmod{4}$ or $\frac{n-1}{2}+r$ is an even integer.
\end{proposition}  
\begin{proof} Let $F$ be a $2$-regular graph of order $2n+1$ such that
\begin{enumerate}[leftmargin=15pt]
  \item\label{cond1} $V(F) = \mathbb{Z}_{2n}\ \cup\ \{\infty\}$,
  \item\label{cond2} $\Delta F \supset \mathbb{Z}_{2n}\setminus\{0\}$,
  \item\label{cond3} $F+n=F$.
\end{enumerate}
and let $C_{\infty}$ denote the cycle of $F$ through $\infty$. 
By assumption, $C_{\infty}$ has length $3$, and
by  conditions \ref{cond1} and \ref{cond3} we have that $C_{\infty}+n=C_{\infty}$; hence
$C_{\infty} = (\infty, g, g+n)$ for some $g\in \mathbb{Z}_{2n}$. 

Let $C_1, C_2, \ldots, C_{u}$ be the list of the cycles in $F$ distinct 
from $C_{\infty}$, with $C_{i}=(c_{i,1}, c_{i,2},  \ldots, c_{i,\ell_i})$ for every $1\leq i \leq u$. 
By condition \ref{cond3}, we can assume without loss of generality that
$C_i+n = C_i$ when  $1\leq i\leq  s$, and 
$C_i + n = C_{i+t}$ when $s+1 \leq i\leq  s+t$,  where $u=s+2t$.
Hence, for $1\leq i \leq s$ we have that $\ell_i$ is even, and by Remark \ref{rotatedcycles}
 we obtain that $c_{i,j+\ell_i/2} = c_{i,j} + n$
for every $1\leq j\leq \ell_i/2$. Now set $m_i=\ell_i/2$ when $1\leq i\leq s$, 
otherwise set $m_i=\ell_i$. Also, let $d_{i,j} = c_{i, j+1} - c_{i,j}$ 
(where the subscript $j$ is computed modulo $\ell_i$) and set $D_i=\{d_{i,j}\mid 1\leq j\leq m_i\}$
for every $1\leq i\leq s+t$. Considering that any translation preserves the differences, we have that
\[
   d\in\Delta C_i,\;\; \text{if and only if}\;\; d\in \pm D_i.
\]
for every $1\leq i\leq s+t$. By recalling condition \ref{cond2}, and considering that $\Delta C_{\infty}=\{\pm n\}$ and $\sum_{i=1}^{s+t} m_i = n-1$,
it follows that $\mathbb{Z}_{2n}\setminus\{0,n\} = \bigcup_{i=1}^{s+t} (\pm D_i)$. 
Also, since $c_{i,1}+\sum_{j=1}^{m_i}d_{i,j}=c_{i, m_i+1}$, we have that 
$\sum_{j=1}^{m_i} d_{i,j}=n$ when $1\leq i\leq s$, otherwise $\sum_{j=1}^{m_i} d_{i,j}=0$; hence
$\sum_{i=1}^{s+t}\sum_{j=1}^{m_i} d_{i,j}=sn$. Finally, considering that $\mathbb{Z}_{2n}\setminus\{0,n\}$ contains 
$2\lfloor \frac{n}{2} \rfloor$ odd integers and $-x\neq x$ for every $x\in \mathbb{Z}_{2n}\setminus\{0,n\}$,
it follows that $\bigcup_{i=1}^{s+t} D_i$ contains exactly $\lfloor \frac{n}{2} \rfloor$ odd integers, therefore
\[
sn\equiv \left\lfloor \frac{n}{2} \right\rfloor \pmod{2}.
\] 
If $n$ is even, then $n\equiv 0 \pmod{4}$. 
If $n$ is odd, then $\frac{n-1}{2}\equiv s\pmod{2}$. Denoting by $s'$ the number of 
even length cycles in $\{C_{s+1}, C_{s+2}, \ldots, C_{s+t}\}$ and recalling that $C_{i}\neq C_i+n\in F$ for 
$s+1\leq i\leq s+t$, it follows that the total number of even length cycles
in $F$ is $r=s+2s'$, hence $\frac{n-1}{2}\equiv r\pmod{2}$, that is, $\frac{n-1}{2} + r$ is even, and the assertion is proven.
\end{proof}

\Prop \ref{prop:1rot} tells us how to construct \OneRot solutions of order $2n+1$. These can then be used, following 
\Prop \ref{prop:2pyr}, to construct solutions of order $2n+2$.
Finally, Propositions  \ref{prop:1Rotnecessary1} and \ref{prop:1Rotnecessary2} give us necessary conditions for a 
\OneRot solution to exist. \\

We use the above results to construct \OneRot solutions to $OP(F)$ whenever $F$ is a $2$-regular graph
of order $4t+1$, thus $n=2t$ and $t \in \mathbb{N}$. Equation (\ref{eq:1Rot_Necessary}) defines $F$ as the graph
containing $u_i$ cycles of odd length $\ell_i$, and $w_j$ cycles of even length $m_j$, for every $i\in\{1,2,\ldots, h\}$
and $j\in\{1,2,\ldots, k\}$. Recalling the necessary condition in Proposition \ref{prop:1Rotnecessary1}, we have
\begin{eqnarray}
F = [^{u_1}\ell_1, \ldots,^{u_h}\ell_h,^{w_1}m_1, \ldots,^{w_k}m_k]	\label{eq:1Rot_newF} &:&|\{i\mid u_i \; \text{is odd}\}| = 1. \label{eq:1Rot_Necessary}
\end{eqnarray}
The graph $F$ must also satisfy  Equation (\ref{eq:1Rot_NewPrepoNecessary}), which implements \Prop \ref{prop:1Rotnecessary2}.

\begin{equation}
\begin{split}
 \exists! i : (\ell_i=3 \land u_i\; \text{is odd}) \Rightarrow  \\ 2t\equiv0\pmod{4}\;\; \lor \;\;    
 \left(\frac{2t-1}{2}+\sum^k_{i=1} w_i\right) \equiv0\pmod{2}. \label{eq:1Rot_NewPrepoNecessary}		
\end{split}
\end{equation}

The symmetries stated in \Rem \ref{rotatedcycles} reduce the labeling problem on $F$ to a simpler one on a new graph $F^*$, the asymmetric version of $F$, which can be seen as the union of $2$ subsets, namely the set of paths 
($\mathcal{P}$) and the set of cycles ($\mathcal{C}$).
To better describe the structure of $F^*$, we assume without loss of generality that $u_1$ is odd, and the remaining $u_i$ are even. Hence we can write 
$u_1 = 2a_1 +1$, $u_i=2a_i$ for every $i\in\{2,3,\ldots, h\}$, and $m_j=2\mu_j$ for every 
$j\in \{1,2,\ldots, k\}$. Thus Equation (\ref{eq:1Rot_FStar}) describes the structure of the reduced graph $F^*$,
\begin{align}
\label{eq:1Rot_FStar}
& F^*= \mathcal{C} \cup \mathcal{P},\\
& \mathcal{C}=\left[(\ell_1+1)/2,^{a_1} \ell_1, ^{a_2}\ell_2, \ldots, ^{a_h} \ell_h\right], \\
& \mathcal{P}=[[^{w_1} \mu_1, \ldots,^{w_k}\mu_k]]. 
\end{align}
where $\mathcal{P}$ is the graph containing $w_j$ paths with $\mu_j$ edges, for every $j\in \{1,2,\ldots, k\}$, and they are pairwise vertex-disjoint.
$\mathcal{C}$ also contains a cycle with $(\ell_1-1)/2$ nodes, namely the one with infinity.

In a more general and intuitive way, the rationale for $F^*$ is as follows. On the one hand, we consider the cycles of odd length in $F$. If there is an even number $u_i \equiv0 \pmod{2}$ of cycles of odd length $l_i$, these contribute in $F^*$ with $u_i/2$ cycles of length $l_i$ in $\mathcal{C}$. The only leftover cycle is of odd length $u_{\infty}$ (since the number of nodes is $4t+1$) is the one with $\infty$, which contributes with an open chain of length $(u_{\infty}-1)/2$ in $\mathcal{P}$. On the other hand, an even number $w_j\equiv0 \pmod{2}$ of cycles of even length $m_j$ are represented in $F^*$ as $w_j/2$ cycles of length $m_j$ in $\mathcal{C}$, similar to an even number of odd cycles. Any cycle of even length $m_j$ with single multiplicity $w_j=1$ becomes a chain of length $m_j/2$ in $\mathcal{P}$.

Note that the number of edges of $F^*$ is $2t-1$.
We seek to determine a vertex labeling of $F^*$ with the elements of $\mathbb{Z}_{4t}$ such
that 
\begin{enumerate}[leftmargin=15pt]
  \item $V(F^*)$ contains exactly one element in $\{x, x+2t\}$, for every $x\in \mathbb{Z}_{4t}$, 
  \item $\Delta F^* = \mathbb{Z}_{4t}\setminus\{0, 2t\}$. 
\end{enumerate}
This labeling of the vertices of $F^*$ leads to a labeling of $F$ satisfying \Prop \ref{prop:1rot}, and hence to a solution for $OP(F)$ (see Figure \ref{img:FStar}).

\subsection{ {(Almost)} \texorpdfstring{\TwoRot}{2-rotational} solutions to \texorpdfstring{$OP(F)$}{OP(F)}}
\label{sec:2rrot}
Here we describe a method to construct solutions to $OP(F)$ in all cases where there is no \OneRot solution and, in particular, when the necessary conditions of Propositions \ref{prop:1Rotnecessary1} and 
\ref{prop:1Rotnecessary2} do not hold.

Let $v=2n+1$ be a positive integer, and  
set $V = \big(\{0,1\} \times \mathbb{Z}_{n}\big)\ \cup\ \{\infty\}$. 
For every subgraph $\Gamma$ of $K_V$ and for every $i,j\in\{0,1\}$, let $\Delta_{ij} \Gamma$ be the list of $(i,j)$-differences of $\Gamma$ defined below:
\begin{eqnarray}
\label{2Rot:Differences}
  \Delta \Gamma_{ij} = \big\{x-y \;\big|\; \e{(i,x)}{(j,y)}\in E(\Gamma\setminus\{\infty\})\big\}.
\end{eqnarray}
For every $g\in \mathbb{Z}_{n}$ we denote by $\Gamma+g$ the graph with vertex set $V$ obtained from $\Gamma$ by replacing each vertex $(i,x)$  with $(i,x+g)$. 

The following result provides sufficient conditions for the existence of a solution to $OP(F)$.

\begin{proposition}\label{prop:2rot} Let $F=[\ell_1, \ell_2 \ldots,\; \ell_s]$ be a $2$-regular graph of order 
$2n+1$ satisfying the following conditions:
\begin{enumerate}[leftmargin=15pt]
  \item\label{2rot:cond1} $V(F) = \big(\{0,1\} \times \mathbb{Z}_{n}\big)\ \cup\ \{\infty\}$,
  \item\label{2rot:cond2} the vertices adjacent to $\infty$ are of the form $(0,x_0)$, $(1,x_1)$ for some $x_0,x_1\in \mathbb{Z}_n$,
  \item\label{2rot:cond3} if $n$ is odd, then $\Delta_{00} F = \Delta_{11} F = \mathbb{Z}_n\setminus\{0\}$ and 
  $\Delta_{01} F = \mathbb{Z}_n$, 
  \item if $n$ is even, then: \label{2rot:cond4}
  \begin{enumerate}
    \item\label{2rot:cond4.1} $F$ contains the path $P=\lfloor (0,0), (0,n/2), (1,n/2), (1,0)\rfloor$ in one of its cycles,
    \item\label{2rot:cond4.2}  $\Delta_{ij} (F - P) = \mathbb{Z}_n\setminus\{0,n/2\}$ for every $(i,j)\in\{(0,0),(0,1),(1,1)\}$.
  \end{enumerate}
\end{enumerate}

Then there exists a solution of $OP([\ell_1, \ell_2 \ldots,\; \ell_s])$. 
Furthermore,
if $C$ is an $\ell_1$-cycle of $F$ such that $\Delta_{01} C$ contains an integer distinct from $n/2$,
then there exists a solution
to $OP([\ell_1+1, \ell_2 \ldots,\; \ell_s])$.
\end{proposition}
\begin{proof}
\label{proof:2rot}
Let $\mathcal{F}=\{F+g\mid g\in[1,n]\}$ when $n$ is odd, otherwise 
let $\mathcal{F}=\{F+g, F^*+(n/2+g)\mid 1\leq g\leq n/2\}$, where $F^*$ is the $2$-regular graph obtained by replacing the path 
$P$ (which is contained in $F$ by condition \ref{2rot:cond4.1}) with $P^* = \lfloor (0,0), (1,n/2), (0,n/2), (1,0)\rfloor$. It is important to notice that in this case
\begin{equation}\label{FvsF*}
  \text{$F -  P = F^* - P^*$}.
\end{equation}
We claim that 
$\mathcal{F}$ is a solution of $OP(F)$. By condition \ref{2rot:cond1} and considering that the total number of edges (counted with their multiplicity) covered by 
$\mathcal{F}$ is $n|F| = n(2n+1) = |E(K_{2n+1})|$, to prove the assertion it is enough to show that every edge of $K_V$, with $V=\big(\{0,1\} \times \mathbb{Z}_{n}\big)\ \cup\ \{\infty\}$, is contained in some $2$-factor of $\mathcal{F}$.

Denoting with $(0,x_0)$ and $(1,x_1)$ the neighbours of $\infty$ in $F$ (condition \ref{2rot:cond2}), we have
that $\lfloor \infty, (i, a)\rfloor\in F-x_i + a$ for every $(i,a)\in V \setminus\{\infty\}$.
By recalling that \eqref{FvsF*} holds when $n$ is even, it follows that every edge of $K_V$ 
incident to $\infty$ belongs to some $2$-factor of $\mathcal{F}$.

Now let $(i,a)$ and $(j,b)$ be two distinct vertices of $V \setminus\{\infty\}$ such that 
$a-b\neq n/2$ for $n$ even.
By conditions  \ref{2rot:cond3} and \ref{2rot:cond4.2}, there exists an edge of $F$, 
say $\lfloor(i,a'), (j,b')\rfloor$ such that $a'-b'=a-b$. It follows that
$\lfloor(i,a), (j,b)\rfloor = \lfloor(i,a'), (j,b')\rfloor + (b-b') \in F+(b-b')$. 
By taking into account \eqref{FvsF*} for $n$ even, we have that $\lfloor(i,a), (j,b)\rfloor$ belongs to some $2$-factor of $\mathcal{F}$.
It is not difficult to check that every edge of the form 
$\lfloor(i,a), (j,a+n/2)\rfloor$, with  $1\leq a\leq n/2$, is contained in $P+a$ or $P+(n/2+a)$. Hence every edge of $K_V$ is contained in some 
$2$-factor of $\mathcal{F}$ which is therefore a solution to $OP(F)$.

Now let $C=(c_0, c_1,\ldots, c_{\ell_1}-1)$ be the $\ell_1$-cycle of $F$ such that $\Delta_{01} C$ contains an element
distinct from $n/2$; in other words, $C$ contains an edge of the form $\lfloor (0,y_0), (1, y_1)\rfloor$ with 
$y_0-y_1\neq n/2$; hence, this edge does not belong to $P$.
Without loss of generality, we can assume that $c_0= (0,y_0)$ and $c_1=(1,y_1)$. 

Set $H$ and $H^*$ be the $2$-regular graphs of order $2n+2$ obtained from $F$ and $F^*$, respectively, by replacing the
edge $\lfloor c_0, c_1 \rfloor$ with the $2$-path $\lfloor c_0, \infty', c_1 \rfloor$, where $\infty'\not \in V$.
Also, $I=\{\lfloor\infty, \infty'\rfloor\}\cup \{\lfloor c_0+g, c_1+g \rfloor \mid 1\leq g\leq n\}$ is clearly a $1$-factor of 
$K_{2n+2}= K_{V\cup\{\infty'\}}$.
Finally, 
let $\mathcal{H}=\{H+g\mid g\in[1,n]\}$ when $n$ is odd, otherwise 
let $\mathcal{H}=\{H+g, H^*+(n/2+g)\mid 1\leq g\leq n/2\}$. 

We claim that $\mathcal{H}$ is a solution to 
$OP([\ell_1+1, \ell_2 \ldots,\; \ell_s])$.
Since $C$ is also a cycle of $F^*$ for $n$ even, both $H$ and $H^*$ are $2$-regular graphs of 
$K_{2n+2}$ isomorphic to $[\ell_1+1, \ell_2 \ldots,\; \ell_s]$. Also, considering that
$\mathcal{F}$ is a $2$-factorization of $K_V$, every edge of $K_{V\cup\{\infty'\}} - I$ not incident to $\infty'$
is contained in some $2$-factor of $\mathcal{H}$. 
Since $H - P = H^* - P^*$ and $\lfloor \infty', (i,a)\rfloor = \lfloor \infty', c_i\rfloor + (a - y_i)$, it follows that
every edge incident to $\infty'$ belongs to some $2$-factor of $\mathcal{H}$, 
therefore $\mathcal{H}$ provides a solution to $OP([\ell_1+1, \ell_2 \ldots,\; \ell_s])$.
\end{proof}

A factorization $\mathcal{F}$ of $K_{2n+1} = K_{V}$, with $V=(\{0,1\}\times \mathbb{Z}_{n})\cup \{\infty\}$, 
constructed as in Proposition \ref{prop:2rot},  {when $n$ is odd}, is called \TwoRot, since 
the permutation group $G=\{\tau_g \mid g\in \mathbb{Z}_{n}\}$ of $V$, where $\tau_g$ fixes $\infty$ and maps 
$(i,x)\in V\setminus\{\infty\}$ to $(i,x + g)$, is an automorphism group of $\mathcal{F}$ whose action on 
$V\setminus\{\infty\}$ creates two orbits of size $n$. 

The idea of constructing $2$-factorizations of $K_{2n+1}$ with $n$ even, as described in Proposition \ref{prop:2rot}, was first presented in \citep{Deza2010}.

\section{Solving instances of the Oberwolfach Problem}
\label{sec:solOP}

%---------------------------------------------
%----------------1-ROTAZIONALI----------------
%---------------------------------------------
\subsection{Computing \texorpdfstring{\OneRot}{1-rotational} solutions}
\label{sub:Solvin1Rot}

Recalling the content of Section \ref{sec:1rot}, we propose a linear-time algorithm that implements \Prop \ref{prop:1Rotnecessary1} and the related \Rem \ref{rotatedcycles}, and reduces $F$ to $F^*$. Afterward, $CP$ solves the problem over $F^*$, and therefore the labeling of $F$ is retrieved.\\
Algorithm \ref{1RotAlg} starts by reducing $F$ to $F^*$ with lines (5-8), where the only unpaired cycle ($u_i =1 \mod{2}$ as of Equation \ref{eq:1Rot_Necessary}) of odd length $l_i$ reduces to a cycle of length $(l_i-1)/2$ in $F^*$ (symmetry of case $a$). This latter cycle contains node $\infty$. In lines (9-12), a pair of 2 isomorphic ($u_i>1$) cycles of odd lengths $l_a$ and $l_b$ is reduced to a single cycle of length $l_i=l_a$ in $F^*$ (symmetry of case $c$). In the same way, in lines (15-17),  a pair of 2 isomorphic ($w_i>1)$ cycles of even length $m_a=m_b$ is reduced to a single cycle of length $m_i=m_a$ in $F^*$ (symmetry of case $c$). In lines (18-21), the remaining unpaired cycles ($w_i=1$) of even lengths $m_i$ are reduced to open chains of length $m_i/2$ in $F^*$ (symmetry of case $b$). We may derive $F$ from $F^*$ by an inverse constructive process.

\begin{algorithm}[!ht]
\scriptsize
\begin{algorithmic}[1]
\State \textbf{Input:} The original graph $F$
\State \textbf{Output:} The transformed version of the graph, $F^*$
\State {infinite=$false$, and $V_{CP}=D_{CP}$=0}
\For{$i$ in $u_i$}  \Comment{Iterate through odd-length cycles}
\If{$u_i$ $\equiv$ 1 \; (mod\;2) \&\& infinite=false }  \Comment{The cycle with infinite}
		\State{Put a chain of length $(l_i-1)/2$ in $F^*$ in position $0$;}
		\State{$u_i--;\quad V_{CP}+=(l_i-1)/2;\quad D_{CP}=l_i/2-2;$}
		\State{infinite=$true$;}
\Else \Comment{$u_i > 1$. Pair of odd-length cycles. Simplify one.}
		\State{Put a cycle of length $l_i$ in $F^*$;}
		\State{$u_i=u_i-2;\quad V_{CP}+=l_i;\quad D_{CP}=l_i;$}
\EndIf
\EndFor
\For{$i$ in $w_i$}  \Comment{Iterate through even-length cycles.}
\If{$w_i > 1$ } \Comment{Pair of even-length cycles. Simplify one.}
		\State{Put a cycle of length $m_i$ in $F^*$;}
		\State{$w_i=w_i-2;\quad V_{CP}+= m_i;\quad D_{CP}=m_i;$}
\Else{ $w_i=1$} \Comment{Treat the cycle as a chain of half length.}
		\State{Put a chain of length $m_i/2$ in $F^*$;}		\State{$w _i--;\quad V_{CP}+=m_i/ 2;\quad D_{CP}=m_i/2;$}
\EndIf
\EndFor
\State \textbf{return} $F^*$
\end{algorithmic}
\caption{\textbf{Algorithm 1: Reducing to $F^*$ }}
\label{1RotAlg}
\end{algorithm}

Following the reduction, the $F^*$ labeling problem \FStar \; aims at finding a labeling for $F^*$ so that there is a solution to the complete $OP(F)$.

\begin{problem}[\FStarName]\label{problem:F_labeling} Let $F^*=(V,E)$ be a graph of order $|V(F^*)|=\gamma+1$, and let $v_i \in V$ be an element in $G\cup\{\infty\}$ where $G=\mathbb{Z}_{2\gamma}$. Also, let $F^*=O\cup C$ with $O$ set of open chains and $C$ set of closed cycles. 
For each node $i \in \Bar{V}=V\backslash\{\infty\}$  the \FStarName \; Problem asks to assign a label $n_i \in G$ so that the following conditions hold:
\begin{enumerate}
	\item Each element in the set $\tilde{V}$ receives a unique label from $G$. \\
	 $n_{\alpha} \neq n_{\beta} \, \text{for all } \, n_{\alpha}, n_{\beta} \in \Bar{V}$, \label{Fstar:cond1}
	\item The set $\Bar{V}$ contains either label $n_{\alpha}$ or its $\gamma$-translated label.\\
	$n_{\alpha} \in \Bar{V} \lor n_{\beta}=n_{\alpha}+\gamma \Mod{2\gamma} \in \Bar{V}$ with $n_{\alpha}\in\mathbb{Z}_{2\gamma}$\label{Fstar:cond2},
	\item  $\Delta F^*$ has cardinality $\lambda=2\gamma-2$ and contains all the elements in $G \backslash \{0,\gamma\}$ with multiplicity $1$.
	$|\Delta F^*|=2(\gamma-2) \land \Delta F^* = G\backslash \{0, \gamma\}$. \label{Fstar:cond3}

\end{enumerate}
\end{problem}

The corresponding $CP$ model (\ref{eq:1ROT_SetV}-\ref{eq:1ROT_AllDifferentD}) describes \FStar.

We remark that \textit{alldifferent} and \textit{card} are typical $CP$ operators on arrays of elements \citep{Beldiceanu2010}. The first operator requires all array elements to exhibit different values. The second one, which takes an additional integer argument value $c$, imposes the cardinality of the integer value $c$ in the given array. 

\begin{align}
  V =\{  n_i \;|\; n_i \in G \}   ,	\label{eq:1ROT_SetV}\\
\texttt{alldifferent(V)} \qquad dom(V)=[0,2\gamma),   \label{eq:1ROT_AlldifferentV}\\
\footnotesize{\texttt{card($V| n_i$)}+\texttt{card($V|(n_i+\gamma\Mod{2\gamma}) $)} =1} \qquad \text{for all } n_i \in \mathbb{Z}_\gamma, \label{eq:1ROT_CardinalityV}\\
 D=dC \cup dO,    \label{eq:1ROT_DeltaSet} 	\\
dC=\{(n_{\alpha} - n_{\beta}  \Mod{2\gamma}) \} \qquad \text{for all } \alpha,\beta \in V \land \lfloor \alpha,\beta \rfloor \in C, \label{eq:1ROT_deltaE} \\
dO=\{ \omega_i-\phi, \phi-\omega_i \Mod{2\gamma}\} \qquad \text{for all } o_i =[\omega_1,\dots,\omega_i] \in O, \label{eq:1ROT_deltaO} \\
\qquad \phi=\omega_1+\gamma \Mod{2\gamma},  \nonumber \\
\texttt{alldifferent(D)} \qquad dom(D)=(0,2\gamma)\backslash\{\gamma\}. \label{eq:1ROT_AllDifferentD}
\end{align}

Equations (\ref{eq:1ROT_SetV})-(\ref{eq:1ROT_AlldifferentV}) enforce Condition (\ref{Fstar:cond1}) in  \FStar, while (\ref{eq:1ROT_CardinalityV}) enforces Condition (\ref{Fstar:cond2}) for $\gamma$-translated labels. The difference-set is split into two subsets, as in Equation (\ref{eq:1ROT_DeltaSet}): $dC$ in (\ref{eq:1ROT_deltaE}) for the edges over $F^*$, and $dO$ for open chains in (\ref{eq:1ROT_deltaO}). The virtual label $\phi$ is reported in the latter subset. Finally, Constraint (\ref{eq:1ROT_AllDifferentD}) enforces Condition (\ref{Fstar:cond3}) in  \FStar. Therefore, the problem of labeling $F^*$ collapses to a feasible assignments of set $V$, represented in (\ref{eq:1ROT_SetV}).

\begin{example}
(referring to Problem \ref{problem:F_labeling})
	Consider an $OP(F=[5,^23,^24,6])$ of order $4t+1=25$ with $t=6$. $F$ reduces to $F^*$, and the simplified instance is  $OP(F^*=[3_{\infty},3,4,3_c])$ where $O=[3_c]$ and $3_{	\infty}$ is the cycle with $\infty$. Therefore $\gamma=12$ and nodes in $V(F^*)$ acquire their labels from $\mathbb{Z}_{24}\cup\{\infty\}$. Figure \ref{img:FStar} represents the reduced $F^*$, with the virtual node $12$. $\Delta F^*=G\backslash\{0,12\}$, hence the labeling is a valid solution for the \FStar. Figure \ref{img:FStar_sol} represents the corresponding  labeling for $F$.
\end{example}
\begin{figure}[!ht]
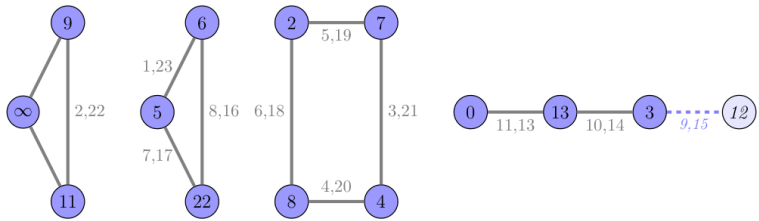

\centering
  \Includiamolo{width=0.8\textwidth}{OP_FStar\Colorato}
      \caption{$F^*$ instance for $F=[5,^23,^24,6]$.}
    \label{img:FStar}
\end{figure}

\begin{figure}[!ht]
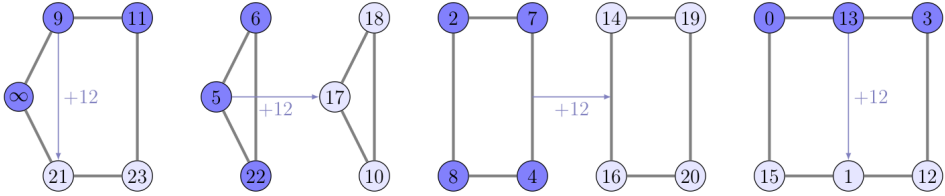

\centering
  \Includiamolo{width=\textwidth}{OP_FStar_Complete\Colorato}
    \caption{$F$ instance derived from $F^*=[3_{\infty},3,4,3_c]$.}
    \label{img:FStar_sol}
\end{figure}

A $2-factor$ $F$ of order $4t+1$ which generates a \OneRot solution for $OP(F)$ satisfies \Prop \ref{prop:1rot}, \ref{prop:1Rotnecessary1}, and \ref{prop:1Rotnecessary2}. Correspondingly, a solution of order $4t+2$ can be derived from $F$, according to \Prop \ref{prop:2pyr}, and its  polynomial-time computation works as follows.
Given $F=[l_1,l_2,...,l_s]$ and $F'=[l_1+1,l_2,...,l_s]$, a new node $\infty'$ joins the cycle $l_1$. The new node replaces a path $P=\lfloor c_1,c_2\rfloor \in l_1$ with a new path $P^*=\lfloor c_1,\infty',c_2\rfloor$ in the cycle $l_1$. Therefore, the difference-set of $F'$ omits values $c_1-c_2$ and $c_2-c_1$ (in modulo). For our computational tests, node $\infty'$ is arbitrarily inserted between two nodes $c_1,c_2$ so that $c_1-c_2=c_2-c_1 \mod{2t}$. Solutions of order $4t+2$ require the same computational effort as \OneRot instances of order $4t+1$. Moreover, we highlight how multiple instances of order $4t+2$ originate from the same instance of order $4t+1$ (see Example \ref{example:1Rot_Higher}).
\\
\begin{example}
\label{example:1Rot_Higher}
	Consider two instances of order $4t+2$, such as $OP(19,4,3)$ and $OP(18,5,3)$ with $t=6$. Both the instances originate from  $OP(18,4,3)$, of order $4t+1$.
\end{example}

%---------------------------------------------
%----------------2-ROTAZIONALI--A-------------
%---------------------------------------------

\subsection{Computing (almost) \texorpdfstring{\TwoRot}{2-rotational} (\texorpdfstring{$n$}{n} odd)}
\label{sub:Solving2Rot}
This class of solutions derives from \Prop \ref{prop:2rot} (see also Section \ref{sec:2rrot}). Since that proposition distinguishes between odd and even values of $n$, we present the approach for odd values of $n$, and discard Condition (\ref{2rot:cond4}) of \Prop \ref{prop:2rot}. For even values of $n$, see the Appendix (\ref{sec:appendix}).\\

Given the \emph{2-regular} graph $F=(V,E)$ of order $2n+1$, recalling that $n$ is odd, we write $n=4t+3$ where $t \in \mathbb{N}$.
The set $V = (\{0,1\} \times G)\ \cup\ \{\infty\}$ represents nodes labels, where $G$ is the additive group $\mathbb{Z}_{2t+1}$. Without loss of generality, $\infty$ lies in the longest cycle of $F$. 
There are $3$ difference-sets, as Condition \ref{2rot:cond3} of \Prop \ref{prop:2rot} states. Each difference $\lfloor \alpha,\beta \rfloor \in E(F)$ goes into a set depending on $\{0,1\}$ labels of both $\alpha$ and $\beta$. We propose a two-step approach that initially determines the first labels and then the second ones. Once the first labels are determined, the problem resembles a \OneRot problem where the decision variables are in a set of integers. On the other side, there are $3$ difference-sets, as described by Equation (\ref{2Rot:Differences}).
The first-step solution provides information about the type of edge (e.g., $\lfloor \alpha=(i,a),\beta=(i,b) \rfloor \;|\; a,b \in G$ is in the difference-set $\Delta F_{ii}$), and the second-step exploits such knowledge.

The \BLPName \; \BLP is the first-step subproblem, and asks to label each node $i \in V(F)\backslash\{\infty\}$ with a coordinate $c_i \in \{0,1\}$, namely the \textit{binary label}. Once the solution of  \BLP is given, the \GLPName \; \GLP seeks to assign a second coordinate $n_i \in G$, namely the \textit{group label}, to each node, so that Condition \ref{2rot:cond3} of \Prop \ref{prop:2rot} holds.  Differences of the type 
$\Delta_{01} F$
are directed from $c_\alpha=c_x=0$ to $c_\beta=c_y=1$.
Without loss of generality, the direction is arbitrarily fixed. Besides, 
$\Delta_{00} F$ and $\Delta_{11} F$ contain both the two undirected differences for each edge (e.g., both $\alpha-\beta$ and $\beta-\alpha$).

According to Condition \ref{2rot:cond1} of \Prop \ref{prop:2rot}, there are exactly $2t+1$ nodes for each binary label. Condition \ref{2rot:cond2} states that two nodes adjacent to $\infty$ have different binary labels. Condition \ref{2rot:cond3} defines difference-sets' cardinalities as 
$|\Delta_{00} F|=|\Delta_{11} F|=2t$ while 
$|\Delta_{01} F|=2t+1$.  \BLP formalizes such conditions.\\

\begin{problem}[\BLPName]\label{problem:Binarylabeling} Let $F=(V,E)$ be a \emph{2-regular} graph of order $|V|=4t+3$.  \BLP asks to assign to each node $i \in \Bar{V}=V\backslash\{\infty\}$ a \textit{binary label} $c_i \in \{0,1\}$ so that the following conditions hold:
\begin{enumerate}[leftmargin=15pt]
	\item The two nodes $\alpha, \beta \in \Bar{V}$ adjacent to $\infty$ have different binary labels. \\ $\exists \lfloor \alpha,\infty \rfloor \land \lfloor \beta,\infty \rfloor : c_{\alpha} \neq c_{\beta}$, \label{BinaryProb:cond1}
	\item There are exactly $2t+1$ directed differences (edges) between nodes with different binary labels. \\$| \Delta_{01} F = \{ \alpha-\beta \;|\; \e{(0,\alpha)}{(1,\beta)} \; \text{for all } \alpha,\beta \in G\}|=2t+1$. 
	 \label{BinaryProb:cond2} 
		
\end{enumerate}
\end{problem}
Equations (\ref{eq:2ROT_CSet}-\ref{eq:2ROT_ColoringInfinite}) formulate \BLP in $CP$.

\begin{eqnarray}
C =\{c_i\;|\; c_i  \in \{0,1\} \} \quad &  \text{for all } i \in \Bar{V}, \label{eq:2ROT_ColoringVarSet} \label{eq:2ROT_CSet} \\
dM=\{d_j\;|\; d_j\in \{0,1\} \} \quad &\text{for all } j \in 0,1,\dots,E(F\backslash\{\infty\}), \label{eq:2ROT_ColoringDiffSet} \\
  d_j=\begin{cases}
    1 & \text{if $c_{\alpha}=0, c_{\beta}=1$} \\
    0 & \text{otherwise}.
  \end{cases} & \text{for all } \alpha,\beta \in \Bar{V} \land \lfloor \alpha,\beta \rfloor,\label{eq:2ROT_ColoringDiffElement} \\
\texttt{card($dM|\; 1$)}=2t+1, & \quad \label{eq:2ROT_ColoringCount}	 \\
\texttt{card($C|\; 1$)}=2t+1, & \quad \label{eq:2ROT_ColoringNodesCount1}	 \\
\texttt{card($C|\; 0$)}=2t+1, & \quad \label{eq:2ROT_ColoringNodesCount2}	 \\
c_{\alpha}=1 \land c_{\beta}=0 &\quad  \text{if} \; \lfloor \alpha,\infty \rfloor \land \lfloor \beta, \infty \rfloor \land \alpha \neq \beta. \label{eq:2ROT_ColoringInfinite}
\end{eqnarray}

\begin{example}
	(referring to Problem \ref{problem:Binarylabeling})
	Consider an $OP(5,6)$ of order $4t+3=11$ with $t=2$. An example of binary labels assigned according to the \BLP  is in Figure \ref{img:BLPProblem}).
\end{example}

\begin{figure}[!ht]
\centering
  \Includiamolo{width=0.5\textwidth}{OP_2Rot_1\Colorato}
    \caption{\BLP over $OP(F=[5,6])$.}
    \label{img:BLPProblem}
\end{figure}

Each node $i \in \Bar{V}$ acquires a binary label $c_i$, hence the solution is the set $C$ in Equation \ref{eq:2ROT_CSet}. Each element $d_j \in dM$ (\ref{eq:2ROT_ColoringDiffSet}) is 1 if and only if the oriented edge $\lfloor \alpha,\beta \rfloor$ connects a node $\alpha: c_{\alpha}=0$ with $\beta: c_{\beta}=1$, and $0$ otherwise.
Constraint (\ref{eq:2ROT_ColoringCount}) ensures that Condition (\ref{BinaryProb:cond2}) of \BLP holds, while Constraints (\ref{eq:2ROT_ColoringNodesCount1}) and (\ref{eq:2ROT_ColoringNodesCount2}) bound the cardinality of binary-labeled nodes in $V$. (\ref{eq:2ROT_ColoringInfinite}) implements Condition (\ref{BinaryProb:cond1}) in \BLP by hard-fixing labels of the two nodes adjacent to $\infty$. 
\newline
\newline
Computational solutions for \TwoRot instances led us to understand the underlying structure of \BLP. Correspondingly, it was possible to devise a general polynomial-time algorithm to solve \BLP. The rationale is to search for known patterns and insert parts of solution (e.g,: label a subset of nodes) until the problem reduces to a basic pattern of the form $F[3]$, $F[5,6]$, $F[^53]$, and $F[^35]$.
Algorithm \ref{2RotBLPAlg_Short} in the Appendix (see \ref{apx:BLPAlgo}) presents such a procedure.

Once \BLP returns the binary labels, solving the $OP(F)$ is matter of a group labeling over the additive group $G$. Condition (\ref{2rot:cond3}) of \Prop \ref{prop:2rot} holds on the difference-sets. In analogy with the labeling for \OneRot methods, each group-label $n_i \in G \;|\; i \in V(F)$ translates into two values in (two) difference-set, depending on its binary label. 
Difference-sets (\ref{2Rot:Differences}) fulfill Equations (\ref{eq:2ROT_Delta00}-\ref{eq:2ROT_Delta01}).
\begin{eqnarray}
\displaystyle
\Delta_{00} F=\{n_{\alpha}-n_{\beta} \Mod{2t+1} : \label{eq:2ROT_Delta00}\\
\text{for all } \alpha,\beta \in V(F) \land \lfloor \alpha,\beta \rfloor \land c_{\alpha}=c_{\beta}=0 \} \nonumber \\
\Delta_{11} F=\{n_{\alpha}-n_{\beta} \Mod{2t+1} : \label{eq:2ROT_Delta11}\\   \text{for all } \alpha,\beta \in V(F) \land \lfloor \alpha,\beta \rfloor \land c_{\alpha}=c_{\beta}=1 \} \nonumber \\
\Delta_{01} F=\{n_{\alpha}-n_{\beta} \Mod{2t+1} : \label{eq:2ROT_Delta01} \\  \text{for all } \alpha,\beta \in V(F) \land \lfloor \alpha,\beta \rfloor \land c_{\alpha}=0,c_{\beta}=1 \}  \nonumber
\end{eqnarray}

Condition \ref{2rot:cond2} of \Prop \ref{prop:2rot} infers domains on sets so that the desired $2-factor$  $F$ is a \TwoRot solution for $OP(F)$. Therefore, the \GLPName\;\GLP formalizes \Prop \ref{prop:2rot}.

\begin{problem}[\GLPName]\label{problem:GLP} Let $F=(V,E)$ be a \emph{2-regular} graph of order $|V|=4t+3$. $V(F) = \{\{0,1\} \times G \}\cup \{ \infty \}$ is the set of nodes over $F$, where $G$ is the additive group $\mathbb{Z}_{2t+1}$. For each node $i \in \Bar{V}=V\backslash\{\infty\}$ with its binary label $c_i \in \{0,1\}$, the \GLP asks to assign a group label $n_i \in G$ so that the following conditions hold:
\begin{enumerate}[leftmargin=15pt]
	\item Undirected difference-sets are so that $\Delta_{00} F=\Delta_{11}F=G\backslash\{0\}$, \label{GLP:cond1}
	\item  The directed difference-set is so that $\Delta_{01} F=G$. \label{GLP:cond2}
\end{enumerate}
\end{problem}

Equations (\ref{eq:2ROT_labelingMasterSet}-\ref{eq:2ROT_labelingAlldifferentDiffAB}) formulate the \GLP  with $CP$.
\begin{eqnarray}
& V =\{A\cup B \}, \quad & \quad \label{eq:2ROT_labelingMasterSet}   \\
A=\{a_i \;|\; a_i \in G,c_i=0\} \quad & \quad B=\{b_i\;|\;b_i \in G,c_i=1\} , \label{eq:2ROT_labelingColorSets} \\
\texttt{alldifferent(A)} \quad & \quad dom(A)=[0,2t+1), \label{eq:2ROT_labelingAlldifferentA}	 \\
\texttt{alldifferent(B)} \quad & \quad dom(B)=[0,2t+1), \label{eq:2ROT_labelingAlldifferentB}	\\
dA=\{(a_{\alpha} - a_{\beta}  \mod (2t+1)) \} \quad &\quad  \text{for all } \alpha,\beta \in A \land \lfloor \alpha,\beta \rfloor, \label{eq:2ROT_labelingSetA} \\
dB=\{(b_{\alpha} - b_{\beta}  \mod (2t+1)) \} \quad & \quad \text{for all } \alpha,\beta \in B \land \lfloor \alpha,\beta \rfloor, \label{eq:2ROT_labelingSetB} \\
dAB=\{(a_{\alpha} - b_{\beta}  \mod (2t+1)) \} \quad & \quad \text{for all } \alpha \in A, \beta \in B \land \lfloor \alpha,\beta \rfloor, \label{eq:2ROT_labelingSetAB} \\
\texttt{alldifferent(dA)} \quad & \quad dom(dA)=(0,2t+1), \label{eq:2ROT_labelingAlldifferentDiffA}	\\
\texttt{alldifferent(dB)} \quad & \quad dom(dB)=(0,2t+1) , \label{eq:2ROT_labelingAlldifferentDiffB}	\\
\texttt{alldifferent(dAB)} \quad & \quad dom(dAB)=[0,2t+1) , \label{eq:2ROT_labelingAlldifferentDiffAB}	\\
\texttt{card(dA)}=\texttt{card(dB)}=2t \quad & \quad \texttt{card(dAB)}=2t+1. \label{eq:2ROT_labelingCardinalityDifferences}
\end{eqnarray}

Equation  (\ref{eq:2ROT_labelingMasterSet}) represents the set $V$ as the union of $A$ and $B$, respectively the subset of nodes with binary label $c_i=0$, and $c_i=1$. The solution is a feasible assignment for $V$. Constraints (\ref{eq:2ROT_labelingAlldifferentA})-(\ref{eq:2ROT_labelingAlldifferentB}) force on both $A$ and $B$ domains over $G$.
Difference-sets in (\ref{eq:2ROT_labelingSetA}-\ref{eq:2ROT_labelingSetAB}) rewrite sets in Equations (\ref{eq:2ROT_Delta00}-\ref{eq:2ROT_Delta11}). Finally, constraints and domains in (\ref{eq:2ROT_labelingAlldifferentDiffA}-\ref{eq:2ROT_labelingAlldifferentDiffAB}) enforce Conditions \ref{GLP:cond1} and \ref{GLP:cond2} of \GLP . In particular, Constraint  (\ref{eq:2ROT_labelingCardinalityDifferences}) ensures that difference-sets have the required cardinalities. The \GLP  solution generates a complete solution for $OP(F)$, with roto-translation similar to the ones explained for \OneRot methods. \Prop \ref{prop:2rot} describes how $F$ generates the other $2-regular$ copies.\\

\begin{example}
(referring to Problem \ref{problem:GLP})
	Consider an $OP(F=[5,6])$ of order $4t+3=11$ with $t=2$. Assuming binary labels are assigned, an example of group labels from  \GLP is represented in Figure \ref{img:GLPProblem}.
\end{example}
	
\begin{figure}[!ht]
\centering
  \Includiamolo{width=0.5\textwidth}{OP_2Rot_2\Colorato}
    \caption{\GLP over $OP(F=[5,6])$.}
    \label{img:GLPProblem}
\end{figure}

According to \Prop \ref{prop:2rot}, a solution of order $4t+3$ generates a solution of order $4(t+1)$. The process requires polynomial-time, and is as follows. Starting from the $4(t+1)$ instance, a 2-path $P=[a,i^*,b]$ is selected from the cycle $C_\infty$ (the cycle containing the $\infty$ node), and replaces the edge $\lfloor a,b \rfloor$. The resulting graph is the \emph{2-regular} $F^*$, of order $|V(F^*)|=4t+3$. The \TwoRot approach solves the $4t+3$ instance on $F^*$. Afterwards,the pruned node $i^*$ goes back to $F^*$, so that $F^*$ becomes $F$. Without loss of generality, $i^*$ lies between nodes with different binary labels $i^*$ inside the cycle $C_\infty$, so that $P^*=[\alpha,i^*,\beta]\;|\;c_{\alpha}\neq c_{\beta}$. Node $i^*$ is relabeled as $\infty_2$ while the original $\infty$ becomes $\infty_1$. Therefore, the difference-set 
$\Delta_{01} F$
on $F$ loses the difference $n_\alpha -n_\beta$ (or $n_\beta - n_\alpha$ if and only if $c_\beta=0,c_\alpha=1$).

\section{Experimental Results}
\label{sec:results}

We implemented the proposed algorithms and formulations with Java 1.8 (see Section \ref{sec:package} for code), \textit{IBM ILOG CPLEX} and \textit{CP Optimizer} $12.7$. Tests ran on a \textit{Intel(R) Core i5-3550 @ 3.30GHz} with $4GB$ of RAM, a computer.
\citet{Deza2010} solve instances of much smaller size (order $\le 40$), with undisclosed algorithms running on a dedicated cluster machine \citep{SHARCNET}. The new 20 orders we solve here, from 41 to 60, are  $241200$, that is 22 times greater than those ($10897$) solved in \citep{Deza2010}. Moreover, they generalize $r$-\emph{rotational} rules also with $r \notin \{1,2\}$, while our contribution deals only with $r \in \{1,2\}$.
Table \ref{table:Results} reports computational results for instances with $n \in [40,60]$, and complete solutions are available online (see Section \ref{sec:package}). Timelimits for \GLP and \FStar\; are  $5 \cdot (1 + |V(F)|/50)$ seconds, and $|V(F^*)|/20$ seconds, respectively, while Algorithm \ref{2RotBLPAlg_Short} solves  \BLP in  negligible time. \\
The \TwoRot approach (see \ref{sub:Solving2Rot}) with odd $n$ values (see \Prop \ref{prop:2rot}) solves instances of order $4t+3$. Solutions of order $4t$ directly derive from the $4t+3$ ones.
The \OneRot approach (see \ref{sub:Solvin1Rot}) solves instances of order $4t+1$, and hence $4t+2$. 
While solving orders $4t+1$, we discovered that certain instances do not have a \OneRot solution, and, consequently, we derived \Prop \ref{prop:1Rotnecessary2}. 

The formal proof stemmed after this empirical evidence. The \TwoRot approach with even $n$ values (see \ref{apx:2RotNeven}) solves instances not fulfilling requirements in \Prop \ref{prop:1Rotnecessary2}. 
We generated all the order-dependent partitions of integers $i \in [40,60]$ with at least three or more cycles (tables), since \citet{TRAETTA2013984} provides complete solutions to the two-table $OP$.
Each order ($1^{st}$ column of Table \ref{table:Results}) is tackled by means of \TwoRot and/or \OneRot rules ($3^{rd}$ column). The $time$ ($4^{th}$ column) represents the total time required for the class $OP$, so that all the instances ($5^{th}$ column) have a solution ($6^{th}$ column). Correspondingly, the average time per instance ($7^{th}$ column) is the arithmetic mean. The proposed methodologies solved all the instances. 
Finally, from our tests, \OneRot methods appear to be faster than \TwoRot methods in terms of CPU times, according to Table \ref{table:Results}. 
Also, we can report that single solutions for $OP$ with $n<120$ could be generated in less than $60$ seconds.

\begin{table}[!ht]
\scriptsize
\centering
\begin{tabular}{@{}r|ccrrrr@{}}
\toprule
\# & \textbf{Type} & \textbf{Method} & \textbf{Time (s)} & \textbf{Partitions} & \textbf{Solved} & \textbf{Avg. Time (s.ms)} \\ \midrule
\textbf{40} & 4t & (Derived from 39) & 911 & 1756 & 1756 & 00.519 \\ \midrule
\textbf{41} & 4t+1 &  & 807 & 2056 & 2056 & 00.393 \\
\textbf{} &  & \OneRot & 90 &  & 1433 & 00.063 \\
\textbf{} &  & A \TwoRot & 717 &  & 623 & 01.151 \\ \midrule
\textbf{42} & 4t+2 & (Derived from 41) & 90 & 2418 & 2418 & 00.037 \\ \midrule
\textbf{43} & 4t+3 & A \TwoRot & 2462 & 2822 & 2822 & 00.872 \\ \midrule
\textbf{44} & 4t & (Derived from 43) & 2462 & 3302 & 3302 & 00.746 \\ \midrule
\textbf{45} & 4t+1 &  & 3268 & 3851 & 3851 & 00.849 \\
\textbf{} &  & \OneRot & 1406 &  & 2547 & 00.552 \\
\textbf{} &  & A \TwoRot & 1862 &  & 1304 & 01.428 \\ \midrule
\textbf{46} & 4t+2 & (Derived from 45) & 1406 & 4488 & 4488 & 00.313 \\ \midrule
\textbf{47} & 4t+3 & A \TwoRot & 6348 & 5215 & 5215 & 01.217 \\ \midrule
\textbf{48} & 4t & (Derived from 47) & 6348 & 6072 & 6072 & 01.045 \\ \midrule
\textbf{49} & 4t+1 &  & 5587 & 7033 & 7033 & 00.794 \\
\textbf{} &  & \OneRot & 460 &  & 4417 & 00.104 \\
\textbf{} &  & A \TwoRot & 5127 &  & 2616 & 01.960 \\ \midrule
\textbf{50} & 4t+2 & (Derived from 49) & 460 & 8158 & 8158 & 00.056 \\ \midrule
\textbf{51} & 4t+3 & A \TwoRot & 16705 & 9441 & 9441 & 01.769 \\ \midrule
\textbf{52} & 4t & (Derived from 51) & 16705 & 10920 & 10920 & 01.530 \\ \midrule
\textbf{53} & 4t+1 &  & 18998 & 12600 & 12600 & 01.508 \\
\textbf{} &  & \OneRot & 4246 &  & 7513 & 00.565 \\
\textbf{} &  & A \TwoRot & 14752 &  & 5087 & 02.900 \\ \midrule
\textbf{54} & 4t+2 & (Derived from 53) & 4246 & 14552 & 14552 & 00.292 \\ \midrule
\textbf{55} & 4t+3 & A \TwoRot & 57043 & 16753 & 16753 & 03.405 \\ \midrule
\textbf{56} & 4t & (Derived from 55) & 57043 & 19296 & 19296 & 02.956 \\ \midrule
\textbf{57} & 4t+1 &  & 42700 & 22183 & 22183 & 01.925 \\
\textbf{} &  & \OneRot & 2519 &  & 12557 & 00.201 \\
\textbf{} &  & A \TwoRot & 40181 &  & 9626 & 04.174 \\ \midrule
\textbf{58} & 4t+2 & (Derived from 57) & 2519 & 25491 & 25491 & 00.099 \\ \midrule
\textbf{59} & 4t+3 & A \TwoRot & 105258 & 29241 & 29241 & 03.600 \\ \midrule
\textbf{60} & 4t & (Derived from 59) & 105258 & 33552 & 33552 & 03.137 \\ \bottomrule
\end{tabular}
\caption{Computational results for the OP with $n \in [40,60]$, with more than $3$ cycles per instance}
\label{table:Results}
\end{table}
\newpage

\subsection{\emph{IP} proves the absence of solution for \texorpdfstring{$OP(^23,5)$}{OP(3,3,5)}}
\label{sec:335}
The \textit{Handbook of Combinatorial Design} from \cite{Colbourn1996} states that it is well known that $OP(^23,5)$ has no solution, referring to a conference proceeding of \citet{Rosa2003}. In a different work, \citet{Alspach1989} cite an unpublished paper by \citet{Piotrowski1991}. In the latter, the author self-cites an \textit{unpublished} paper \citep{PiotrowskiUnp} from 1979, where he describes a proof made with the \textit{aid of a computer}. We provide a simple proof of non-existence for $OP(^23,5)$ with an $IP$ formulation.
The $OP(^23,5)$ is the problem of arranging $n=11$ people in $2$ tables of $3$ and $1$ table of $5$ for $M=5$ meals. Each person has a label in $\mathbb{Z}_{11}$. The $IP$ formulation enumerates every feasible combination of labels for tables of $3$ (triplets) and tables of $5$ (5-sets). Afterwards, $IP$ seeks to select for $M=5$ meals, one 5-set and two triplets so that each node \textit{is seated} next to every other node exactly once during all the meals. There are $\binom{11}{3}=165$ different triplets, with at least one distinct label, and $\binom{11}{5}\cdot{}12=5544$ 5-sets with different adjacencies. The $IP$ formulation in (\ref{eq:OP335_Objective})-(\ref{eq:OP335_SetDef}) models $OP(^23,5)$. 
%We exploit such a formulation since its relaxation efficiently outputs the result.

\begin{eqnarray}
     min_{F,T}& 0  \label{eq:OP335_Objective}\\
	\text{subject to} & \nonumber \\
	\sum_{i \in I} F_{id} =1 & \text{for all } d \in D, \label{eq:OP335_OneFivePerDay}\\
	\sum_{j \in J} T_{jd} =2 & \text{for all } d \in D, \label{eq:OP335_TwoTripletsPerDay}\\
	\sum_{i \in I} F_{id}\cdot fl_{il} + \sum_{j \in J} T_{jd}\cdot tl_{jl} =1 & \text{for all } d \in D, \text{for all } l \in L, \label{eq:OP335_OneLabelPerDay} \\
	\sum_{d \in D}( \sum_{i \in I} F_{id}\cdot fa_{i\alpha\beta} + \sum_{j \in J} T_{jd}\cdot ta_{i\alpha\beta})=1 & \text{for all } \alpha,\beta \in L \land \alpha \neq \beta , \label{eq:OP335_OneAdjacencyOverDays} \\
	F_{id},T_{jd} \in \{0,1\} & \text{for all } i \in I, j \in J, d \in D. \label{eq:OP335_SetDef}
\end{eqnarray}
The binary variables $F_{id} : i \in I=[1,5544]$ and $T_{jd} : j \in J=[1,165]$ with $d \in D=[1,D=5]$ respectively represent all the different 5-sets and triplets over the $5$ meals, respectively. $F_{id}$ and $T_{jd}$ take value of $1$ if and only if the corresponding element - the $i^{\rm th}$ 5-set or $j-th$ triplet - is selected for the $d-th$ day. Coefficients $fl_{il}$ and $tl_{jl}$ are respectively equal to $1$ if the label $l \in L=[1,11]$ is present in the $i^{\rm th}$ 5-set or $i^{\rm th}$ triplet. Coefficients $fa_{i\alpha \beta}$ and $ta_{j\alpha \beta}$ are equal to $1$ if two different labels $\alpha,\beta \in L$ are adjacent in the $i^{\rm th}$ 5-set or $i^{\rm th}$ triplet. 
The model has no objective function, as of in Equation (\ref{eq:OP335_Objective}). Equations (\ref{eq:OP335_OneFivePerDay}) and (\ref{eq:OP335_TwoTripletsPerDay}) enforce the selection of one 5-set and two triplets per each $d \in D$. Constraint (\ref{eq:OP335_OneLabelPerDay}) enforces that each label appears only once for each $d \in D$. Constraint (\ref{eq:OP335_OneAdjacencyOverDays}) enforces that two labels $\alpha,\beta \in L$ are adjacent only one time over all the meals.
The continuous relaxation of  (\ref{eq:OP335_Objective})-(\ref{eq:OP335_SetDef}) finds no solution in less than a second: hence  $OP(^23,5)$ has no solution. 
%Exhaustive formulations for the problem either in $IP$ or $CP$ required significantly more time to solve.
Correspondingly, the following proposition holds.

\begin{proposition}

Let $OP(^23,5)$ be the Oberwolfach Problem with $2$ cycles of length $3$ and a cycle of length $5$. There is no solution to $OP(^23,5)$.
\end{proposition}
%\begin{proof}
%	Proven with the $IP$ model in Equations (\ref{eq:OP335_Objective})-(\ref{eq:OP335_SetDef}).
%\end{proof}
\section{Concluding remarks}
\label{sec:remarks}
$CP$ -- particularly its propagation and inference algorithms -- exploits difference methods for the $OP$ problem by means of 
well-posed formulations. $1$ and \TwoRot methods reduce the complete $OP$ to the problem of labeling a single $2-factor$
and that problem is efficiently solved in $CP$. In particular, Constraint propagation exploits the relations of mutual exclusion between labels to smartly guide the search procedure. Computational results prove the effectiveness of the approach, which provided complete solutions for the $OP$ for $n \in [40,60]$.
Moreover, theoretical results such as \Prop \ref{prop:1Rotnecessary2} and the proof of absence of a solution for $OP(^23,5)$ 
constitute a further outcome of this work.
The complementarity of Combinatorial Design and Combinatorial Optimization and their positive interaction is, in our view, the main stake pointed out here. The contribution is bidirectional: computational evidence helps to deduce theoretical results, and the latter provides models for the former. We believe that approaches similar to the one presented may help to solve other problems in Combinatorial Design Theory.

\subsection{Solutions and code}
\label{sec:package}
We complement the paper with the software written to implement the presented methodologies. To make its use accessible, we provide a GUI interface.
\\
{\footnotesize
\\ The software is available on GitHub at the following repository:
\\ \url{https://github.com/ALCO-PoliTO/TheOberSolver}
\\Full solutions are available at:
\\ \url{ https://github.com/ALCO-PoliTO/TheOberSolver/tree/master/OberResults}. 
\\The $IP$ \textit{DinnerFor11} formulation of $OP(^23,5)$ is available at 
\\\url{https://github.com/ALCO-PoliTO/DinnerFor11}
}

\section*{Acknowledgements}
This work has been partially supported by ``Ministero dell’Istruzione, dell’ Università e della Ricerca" Award \emph{TESUN-83486178370409} finanziamento dipartimenti di eccellenza CAP. 1694 TIT. 232 ART. 6" .
\noindent

\newpage
\section*{Appendix}
\label{sec:appendix}

\subsection*{BLP Algorithm}
\label{apx:BLPAlgo}
The graph $F=[^{n_1}l_1,^{n_2}l_2,...,^{n_a}l_b]$ is described with $n_i$ the number of cycles of length $l_i$, $T=\sum^{a}_i n_i$ the number of cycles, and $t_{M}=\max_b l$ the longest cycle in $F$. 
The input is an unlabeled $F$ and the output is the $BLP$ solution for $F$, namely $F_l$. The Algorithm iteratively adds to the incumbent a partial labeling for a known pattern, and terminates when all the nodes have been labeled and transferred from $F$ to $F_l$. Lines (3-6) iterate through $T$ cycles, and reduce each cycle  $l_i\ge 7$ with a cycle of length at most of $6$, by labeling patterns of $4$ nodes at a time. Cycles with exactly $4$ nodes - as of in lines (7-9), are labeled instantly. Lines (12-18) search for more complex patterns (e,g: $F=[3,5]$). Lines (19-21) label basic patterns in $F$. The order reported in line (14) is binding, and labeled patterns have different orientations depending on the incumbent labeled nodes. The worst-case time complexity of Algorithm \ref{2RotBLPAlg_Short} is $\mathcal{O}(T\cdot t_M)$ with an implementation without Lists. 

\begin{algorithm}[!htb]
\scriptsize
\begin{algorithmic}[1]
\State \textbf{Input:} The graph $F=[^{n_1}l_1,^{n_2}l_2,...,^{n_a}l_b]=[t_1,t_2,...,t_T]$
\State \textbf{Output:} The labeled (colored) graph $F_l$
\For{$t_i$ in $F$} 
\While{$t_i \ge 7$  } 
		\State{$F_l \leftarrow $ last labels of $t_i$ are $[1100]$;$t_i \leftarrow (t_i-4)$ }\Comment{Color last four elements}
\EndWhile
\If{$t_i$ = 4} 
		\State{$F_l \leftarrow $ labels of $t_i$ are $[1100]$; $t_i \leftarrow (t_i-4)$ } \Comment{Color last four elements}
\EndIf
\EndFor
\State{found=$true$}
\While{$found$}  \Comment{Search for known patterns}
\State{$found$=false}
\Comment{The order of search is strictly as specified}
\If{$F$ contains patterns like $F[3,5]$, $F[^45]$, $F[^83]$, $F[^26]$, $F[^23,6]$ } 
	\State{$found$=true}
	\State{$F_l \leftarrow $ labels of $t$ are [Pattern]; $t_i \leftarrow (t_i-Pattern.length)$}
\EndIf
\EndWhile
\If{the remaining problem $t$ in $T_i$ is of the form of $F[3]$,$F[5,6]$,$F[^53]$,$F[^35]$ } 
	\State{$F_l \leftarrow $ labels of $t$ are [Pattern]; $t_i \leftarrow (t_i-Pattern.length)$}
\EndIf \Comment{Check for basic solutions. The order of search is as specified}
\State{\textbf{return} $F_l$}
\end{algorithmic}
\caption{\textbf{Algorithm 2: \BLP Algorithm }}
\label{2RotBLPAlg_Short}
\end{algorithm}

%---------------------------------------------
%----------------2-ROTAZIONALI--B-------------
%---------------------------------------------

\subsection*{(Almost) \TwoRot with $n$ even}
\label{apx:2RotNeven}
The approach to this class of instances is similar to the one presented for odd $n$. The \TwoRot method with even $n$ solves instance not fulfilling \Prop \ref{prop:1Rotnecessary1}.
The methodology is analogous to the one for odd $n$, but Condition \ref{2rot:cond3} from \Prop \ref{prop:2rot} is discarded, while Condition \ref{2rot:cond4} holds. If $n$ is even, $F$ has order of $4t+1$ with $n=2t$. The set  of vertices is $V = (\{0,1\} \times G)\ \cup\ \{\infty\}$, with G the additive group $\mathbb{Z}_{2t}$. 
 \BLP and \GLP slightly vary, according to the Proof of \Prop \ref{prop:2rot}.
In particular, according to Condition \ref{2rot:cond4} of \Prop \ref{prop:2rot}, a cycle of $F$ contains a path of the form $P=\lfloor (0,0), (0,n/2), (1,n/2), (1,0)\rfloor$. The modified\BLP takes into account the path $P$ so that the first two nodes of $P$ take the label $0$, and the last two the label $1$. 
We call \textit{critical paths} all the candidates paths in $F$. Difference-sets, represented Equations (\ref{eq:2ROT_Delta00}-\ref{eq:2ROT_Delta01}), are considered over the graph $F-P$, with modulo $4t$, and their  cardinality is lowered to $2t-2$. For ease of notation, the introduced new sub-problems are the \BLPNameEven\;\BLPEven and the \GLPNameEven\;\GLPEven.

\begin{problem}[\BLPNameEven]\label{problem:BinarylabelingEven} Let $F=(V,E)$ be a \emph{2-regular} graph of order $|V|=4t+1$. The \BLPEven\;asks to assign to each node $i \in \Bar{V}=V\backslash\{\infty\}$ a binary label $c_i \in \{0,1\}$ so that the following conditions hold:
\begin{enumerate}
	\item $\lfloor \alpha,\infty \rfloor \land \lfloor \beta,\infty \rfloor \Longrightarrow c_{\alpha} \neq c_{\beta}$,\label{BinaryProbEven:cond1}
	\item There is at least one critical path $P$ so that \\ 
	$P=\lfloor x,y,z,k \rfloor \;|\; x,y,z,k \in \Bar{V} \land c_x=c_y=0, c_z=c_k=1$,  \label{BinaryProbEven:cond3}
	\item $| \Delta (F-P)_{01} = \{ \alpha-\beta \;:\; \e{(0,\alpha)}{(1,\beta)} \; \forall\; \alpha,\beta \in G\}|=2t-2$.
		\label{BinaryProbEven:cond2}
\end{enumerate}
\end{problem}

The $CP$ model (\ref{eq:2ROT_CSet}-\ref{eq:2ROT_ColoringInfinite}) is modified to fit the additional Condition (\ref{BinaryProbEven:cond3}) for the \BLPEven. Constraints (\ref{eq:2ROT_ColoringCount}-\ref{eq:2ROT_ColoringNodesCount2}) are modified to enforce the new cardinality $(2t-1)$ for both $dM$ and $C$. Moreover, the following additional Constraints hold:

\begin{eqnarray}
A=\{A_i\;|\; A_i\in \{0,1\} \} \quad &\text{for all } i \in PA, \label{eq:2ROT_PathSet} \\
  A_i=\begin{cases}
    1 & \text{if $c_x=c_y=0 \land c_z=c_k=1$ \quad} \\
    0 & \text{otherwise}.
  \end{cases} & PA_i=\lfloor x,y,z,k \rfloor, & \quad \label{eq:2ROT_PathSetDef} \\
& \texttt{card($A_i|1$)} \ge1. & \quad \label{eq:2ROT_PathSetCard}	
\end{eqnarray}

The set $PA$ in (\ref{eq:2ROT_PathSet}) enumerates all combinations of four consecutive nodes in $F$. In Constraints (\ref{eq:2ROT_PathSet}) and (\ref{eq:2ROT_PathSetDef}), each element $A_i \in A$ is set to $1$ if $c_x=c_y=0 \land c_z=c_k=1$, and hence $A_i$ is a critical path. At least one critical path exists as of (\ref{eq:2ROT_PathSetCard}). Once \BLPEven is solved, the \GLPEven labels are determined considering a single critical path $A_i$. If no solution is found for the latter sub-problem, a new critical path $A_j \neq A_i$ induces a different \GLPEven. In terms of \GLPEven, Condition \ref{2rot:cond4.1} induces four specific group-labels on the critical path $A_i$. 

\begin{problem}[\GLPNameEven]\label{problem:GLPEven} Let $F=(V,E)$ be a \emph{2-regular} graph of order $|V(F)|=4t+1$. $V(F) = \{\{0,1\} \times G \}\cup \{ \infty \}$ is the set of nodes over $F$, where $G$ is the additive group $\mathbb{Z}_{2t}$. For each node $i \in \Bar{V}=V\backslash\{\infty\}$ - given the binary label $c_i \in \{0,1\}$ of $V$ and a critical path $P$, the \GLPEven\;asks to assign a label $n_i \in G$ so that the following conditions hold:
\begin{enumerate}
	\item Difference sets are so that \\
	$\Delta_{00}(F-P)=\Delta_{11}(F-P)=\Delta_{01}(F-P)=G\backslash\{0, t \}$, \label{GLPEven:cond1}
	\item  F contains the path $P=\lfloor x,y,z,k \rfloor=\lfloor (0,0), (0,t), (1,t), (1,0)\rfloor$. \label{GLPEven:cond2}
\end{enumerate}
\end{problem}
 
The \GLPEven $CP$ formulation is similar to the one in (\ref{eq:2ROT_labelingMasterSet}-\ref{eq:2ROT_labelingAlldifferentDiffAB}), and the critical-path $P=A_i$ constitutes an additional input.

\begin{eqnarray}
& V =\{A\cup B \}, \quad & \label{eq:2ROTE_labelingMasterSet}   \\
 A=\{a_i \;|\; a_i \in G,c_i=0\} \quad &  B=\{b_i\;|\;b_i \in G,c_i=1\},\label{eq:2ROT_labelingColorSets2} \\
\texttt{alldifferent(A)} \quad &  dom(A)=[0,2t), \label{eq:2ROTE_labelingAlldifferentA}	 \\
\texttt{alldifferent(B)} \quad &  dom(B)=[0,2t), \label{eq:2ROTE_labelingAlldifferentB}	\\
dA=\{(a_{\alpha} - a_{\beta} \mod (2t)) \} \quad &  \text{for all } \alpha,\beta \in A, \lfloor \alpha,\beta \rfloor \notin P, \label{eq:2ROTE_labelingSetA} \\
dB=\{(b_{\alpha} - b_{\beta}  \mod (2t)) \} \quad &  \text{for all } \alpha,\beta \in B, \lfloor \alpha,\beta \rfloor \notin P, \label{eq:2ROTE_labelingSetB} \\
dAB=\{(a_{\alpha} - b_{\beta} \mod (2t)) \} \quad &  \text{for all } \alpha \in A, \beta \in B, \lfloor \alpha,\beta \rfloor \notin P, \label{eq:2ROTE_labelingSetAB} \\
\texttt{alldifferent(dA)} \quad &  dom(dA)=(0,2t) \backslash\{t\}, \label{eq:2ROTE_labelingAlldifferentDiffA}	\\
\texttt{alldifferent(dB)} \quad &  dom(dB)=(0,2t) \backslash\{t\}, \quad \label{eq:2ROTE_labelingAlldifferentDiffB}	\\
\texttt{alldifferent(dAB)} \quad & dom(dAB)=(0,2t) \backslash\{t\}, \label{eq:2ROTE_labelingAlldifferentDiffAB}	\\
n_x=n_k=0,n_y=n_z=t  \quad & \ P=\lfloor x,y,z,k \rfloor, \label{eq:2ROTE_CriticalPath}\\
\texttt{card(dA)}=\texttt{card(dB)}=\texttt{card(dAB)} & \texttt{card(dAB)}=2t-2. \label{eq:2ROTE_labelingCardinalityDifferences}	
\end{eqnarray}

Constraints (\ref{eq:2ROT_labelingColorSets2}-\ref{eq:2ROTE_labelingSetAB} and \ref{eq:2ROTE_labelingCardinalityDifferences}) have different cardinalities and modulo arguments in. The value $n/2=t$ is not in difference-sets of Constraints (\ref{eq:2ROTE_labelingSetA}-\ref{eq:2ROTE_labelingSetAB}). Moreover, edges inside $P$ do not figure in difference-sets. Condition \ref{GLPEven:cond2} of \GLPEven assigns labels for nodes inside the critical path $P$, as  in Constraint (\ref{eq:2ROTE_CriticalPath}).

\bibliographystyle{abbrvnat}
\bibliography{Biblio.bib}
\label{sec:bib}

\end{document}